\documentclass[12pt]{article}
\usepackage[dvipsnames]{xcolor}
\usepackage{tikz}
\usepackage{titlesec}
\titlespacing*{\section}{0pt}{0.5\baselineskip}{0.5\baselineskip}
\titlespacing*{\subsection}{0pt}{0.2\baselineskip}{0.2\baselineskip}
\usepackage[title]{appendix}

\usepackage{xcolor}
\usepackage{empheq}
\usepackage{geometry}
\usepackage[shortlabels]{enumitem}
\setlist[enumerate]{nosep}
\usepackage[doc]{optional}
\usepackage{float}
\usepackage{soul}
\usepackage{graphicx}
\definecolor{labelkey}{rgb}{0,0.08,0.45}
\definecolor{refkey}{rgb}{0,0.6,0.0}
\definecolor{Brown}{rgb}{0.45,0.0,0.05}
\definecolor{lime}{rgb}{0.00,0.8,0.0}
\definecolor{lblue}{rgb}{0.5,0.5,0.99}

 \usepackage{mathpazo}

\usepackage{stmaryrd}
\usepackage{amssymb}
\definecolor{dgreen}{rgb}{0.00,0.49,0.00}
\definecolor{dblue}{rgb}{0,0.08,0.55}
\colorlet{minhblue}{dblue}
\colorlet{mygreen}{dgreen}
\usepackage{hyperref}
\hypersetup{ 
    linktocpage=true,
    colorlinks= true,
    urlcolor = cyan, 
    citecolor = Red,
    linkcolor = dblue
}

\newcommand{\seppfive}{\setlength{\itemsep}{-5pt}}

\newcommand{\thmtit}[1]{{\bf{#1}}}

\providecommand{\siff}{\Leftrightarrow}
\newcommand{\weakly}{\ensuremath{\:{\rightharpoonup}\:}}

\newcommand{\nnn}{\ensuremath{{n\in{\mathbb N}}}}
\newcommand{\knn}{\ensuremath{{n\in{\mathbb K}}}}

\newcommand{\menge}[2]{\big\{{#1}~\big |~{#2}\big\}}

\newcommand{\fenv}[1]%
{\ensuremath{\,\overrightarrow{\operatorname{env}}_{#1}}}
\newcommand{\benv}[1]%
{\ensuremath{\,\overleftarrow{\operatorname{env}}_{#1}}}

\newcommand{\RR}{\ensuremath{\mathbb R}}

\newcommand{\dom}{\ensuremath{\operatorname{dom}}}
\newcommand{\argmin}{\ensuremath{\operatorname{argmin}}}
\newcommand{\prox}{\ensuremath{\operatorname{Prox}}}

\newcommand{\sri}{\ensuremath{\operatorname{sri}}}

\newcommand{\ran}{\ensuremath{\operatorname{ran}}}
\newcommand{\zer}{\ensuremath{\operatorname{zer}}}

\newcommand{\Id}{\ensuremath{\operatorname{Id}}}

\newcommand{\TDR}{T_{\text{\tiny \rm DR}}}
\newcommand{\TPR}{T_{\text{\tiny \rm PR}}}

\newcommand{\bx}{\ensuremath{\mathbf{x}}}

\newcommand{\by}{\ensuremath{\mathbf{y}}}

\DeclareMathOperator*{\minimize}{minimize\ }
\newcommand{\veet}{\ensuremath{{\scriptscriptstyle\vee}}} 

{\begin{list}{}{%
\settowidth{\labelwidth}{\textrm{#1~}}%
\setlength{\leftmargin}{\labelwidth+\labelsep}}}
{\end{list}}
\usepackage{amsthm}
\usepackage[capitalize, nameinlink]{cleveref}
\crefname{enumi}{}{}
\crefname{appsec}{Appendix}{Appendices}
\crefname{equation}{}{equations}
\crefname{chapter}{Appendix}{chapters}
\crefname{item}{}{items}
\newtheorem{theorem}{Theorem}[section]
\newtheorem{lemma}[theorem]{Lemma}
\newtheorem{lem}[theorem]{Lemma}

\newtheorem{prop}[theorem]{Proposition}

\newtheorem{thm}[theorem]{Theorem}
\newtheorem{example}[theorem]{Example}
\newtheorem{ex}[theorem]{Example}
\newtheorem{fact}[theorem]{Fact}
\newtheorem{remark}[theorem]{Remark}

\providecommand{\norm}[1]{\lVert#1\rVert}
\providecommand{\normsq}[1]{\lVert#1\rVert^2}

\providecommand{\innp}[1]{\langle#1\rangle}

\providecommand{\grad}{\nabla}

\providecommand{\RR}{\mathbb{R}}

\providecommand{\ran}{\operatorname{ran}}

\providecommand{\intr}{\operatorname{int}}

\providecommand{\dom}{\operatorname{dom}}

\newcommand{\fix}{\ensuremath{\operatorname{Fix}}}

\providecommand{\Id}{\operatorname{{ Id}}}

\providecommand{\fady}{\varnothing}

\providecommand{\argmin}{\mathrm{arg}\!\min}

\providecommand{\fix}{\operatorname{Fix}}
\providecommand{\ran}{\operatorname{ran}}

\providecommand{\Id}{\operatorname{Id}}

\providecommand{\pt}{{\partial}}

\providecommand{\zer}{\operatorname{zer}}

\providecommand{\fady}{\varnothing}

\providecommand{\ri}{\operatorname{ri}}

\providecommand{\RR}{\mathbb{R}}

  \newcommand*\mybluebox[1]{%
    \colorbox{RoyalBlue!20}{\hspace{1em}#1\hspace{1em}}}

\allowdisplaybreaks 

\begin{document}

%

%

\title{\textsc
A note on the equivalence of operator splitting methods}

\author{Walaa M.\ Moursi\thanks{Stanford 
University, Packard Building, Room 222,
 350 Serra Mall, Stanford, CA 94305, USA
and 
Mansoura University, Faculty of Science, Mathematics Department, 
Mansoura 35516, Egypt. E-mail:
\texttt{wmoursi@stanford.edu}.}
~and Yuriy Zinchenko\thanks{University of Calgary,
Department of Mathematics and Statistics
Math Science Building , 
2500 University Drive NW
Calgary, Alberta, T2N 1N4, Canada.
E-mail: \texttt{yzinchenko@ucalgary.ca}.}}

\date{June 8, 2018}

\maketitle

\begin{abstract}
\noindent
This paper provides a comprehensive 
discussion of the equivalences 
between splitting methods.
These equivalences have been studied 
over the past few decades and, in fact, 
have proven to be very useful.
In this paper, we survey known results 
and also present new ones.
In particular, we provide simplified proofs 
of the equivalence of
the ADMM and the Douglas--Rachford method
and the equivalence of
the ADMM with intermediate update of multipliers and 
the Peaceman--Rachford method.
Other splitting methods are also considered.

%
\end{abstract}
{\small

}

\section{Introduction}

%


Splitting methods 
have become popular in solving convex optimization
problems that involve finding a minimizer of the sum of two 
proper lower semicontinuous convex functions. 
Among these methods are the Douglas--Rachford 
and the Peaceman--Rachford methods introduced 
in the seminal work of Lions and Mercier
\cite{L-M79},  
the forward-backward method
(see, e.g., \cite{Comb04} and \cite{Tseng91}),
Dykstra's method (see, e.g., \cite{BC08}
and \cite{Boyle-Dykstra86}),
and 
the Method of Alternating
Projections (MAP) (see, e.g., \cite{Deutsch01}).

When the optimization problem features 
the composition of one of the functions with
a bounded linear operator,
a popular technique
is the Alternating-Direction Method of Multipliers (ADMM)
(see \cite[Section~4]{Gabay83},
\cite[Section~10.6.4]{CP11}
and also \cite[Chapter~15]{Beck17}).
The method has a wide range of 
applications including large-scale optimization,
machine learning,
image processing and portfolio optimization,
 see, e.g., 
 \cite{Boyd-PCPE11},
 \cite{CP08}
  and 
 \cite{EckThesis}.
A powerful framework to use
ADMM in the more general setting of 
\emph{monotone 
operators}
is developed in the work of 
Brice{\~n}o-Arias
 and 
Combettes 
\cite{B-AC} (see also \cite{Bot17} and \cite{CP}).
Another relatively recent method 
is the Chambolle--Pock 
method introduced in \cite{CP2010}.

Equivalences between splitting methods 
have been studied over the past four decades.
For instance, 
it is known that ADMM is equivalent to 
the Douglas--Rachford method 
\cite{L-M79} (see, also \cite{EckBer})
in the sense that 
with a careful choice of the starting point,
one can prove that the sequences generated 
by both algorithms coincide. (See, e.g.,
\cite[Section~5.1]{Gabay83} and \cite[Remark~3.14]{BK12}.)
A similar equivalence holds between ADMM
(with intermediate update of multiplier)
and Peaceman--Rachford method \cite{L-M79}
 (see \cite[Section~5.2]{Gabay83}). 
In \cite{OV17}, the authors  proved the 
correspondence of Douglas--Rachford
 and Chambolle--Pock methods.

\emph{In this paper,
we review the equivalences between different splitting methods.
Our goal is to present a comprehensive
and self-contained study of these equivalences.
We also present counterexamples that
show failure of some equivalences.}


The rest of this paper is organized as follows:
Section~\ref{sec:2} provides a brief literature review
of ADMM, Douglas--Rachford 
and Peaceman--Rachford methods.
In Sections~\ref{sec:4} and \ref{sec:5}
we explicitly describe the equivalence of  
ADMM (respectively ADMM with intermediate 
update of multipliers) and Douglas--Rachford 
(respectively Peaceman--Rachford) method
introduced by Gabay in \cite[Sections~5.1\&5.2]{Gabay83}.
We provide simplified proofs of these equivalences.
Section~\ref{sec:CP} focuses on the recent
work of O'Connor and Vandenberghe 
concerning the equivalence of 
Douglas--Rachford and Chambolle--Pock 
(see \cite{OV17}). 
In Section~\ref{sec:6}, we provide 
counterexamples which show that,
in general, we cannot deduce 
the equivalence of ADMM to Dykstra's method or to MAP. 

Our notation is standard and follows largely,
e.g., \cite{BC2017}.

\section{Three techniques}
\label{sec:2}
In this paper, we assume that
\begin{empheq}[box=\mybluebox]{equation*}
\label{X:assmp}
X \text{~ and~} Y\text{~are real Hilbert spaces},
\end{empheq} 
 that
\begin{empheq}[box=\mybluebox]{equation*}
\label{fg:assmp}
\text{$
f\colon X\to \left]-\infty,+\infty\right]
$ and $
g\colon Y\to \left]-\infty,+\infty\right]$
are convex lower semicontinuous and proper.}
\end{empheq} 
\subsection*{Alternating-Direction Method of Multipliers (ADMM)}
In the following we assume that
\begin{empheq}[box=\mybluebox]{equation}
\label{eq:cq:L}
\text{ $L \colon Y\to X$ is linear
such that $L^*L$ is invertible},
\end{empheq} 
that
\begin{equation}
  \label{eq:cq:L:i}
\argmin(f\circ L+g) \neq \fady,
\end{equation}
and that
\begin{equation}
  \label{eq:cq:L:ii}
 0\in \sri(\dom f-L(\dom g)),
\end{equation}
where $\sri S$ denotes 
the \emph{strong relative interior} of a subset $S$ of $X$
 with respect to the closed affine hull of $S$.
When $X$ is finite-dimensional we have $\sri S=\ri S$, where $\ri S$
is the \emph{relative interior} of $S$ defined as 
the interior of $S$ with respect to the affine hull of $S$.

Consider the problem
 \begin{equation}
 \label{eq:prob}
\minimize_{y\in Y} f(Ly)+g(y).
 \end{equation}
Note that \cref{eq:cq:L:i}
 and
\cref{eq:cq:L:ii}
imply that (see, e.g., \cite[Proposition~27.5(iii)(a)1]{BC2017})
 \begin{equation}
 \label{eq:connec:2:inc}
\argmin(f\circ L+g)
=\zer( \pt(f\circ L)+\pt g)
=\zer(L^*\circ (\pt f) \circ L+\pt g)\neq \fady.
 \end{equation}
 
In view of \cref{eq:connec:2:inc},
solving \cref{eq:prob} is equivalent to solving the 
inclusion:  
\begin{equation}
\label{eq:inc:inc}
\text{Find $y\in Y$ such that
$0\in L^*(\pt f  (Ly))+\pt g(y)$.}
\end{equation}

%

%

The augmented Lagrangian   associated with
\cref{eq:prob} is the function
\begin{equation}
\mathcal{L}\colon X\times Y \times Y\to \left]-\infty,+\infty\right]
\colon (a,b,u)\mapsto f(a)+g(b)+\innp{u,\mathcal{L}a-b}
+\normsq{\mathcal{L}a-b}.
\end{equation}
The ADMM (see 
\cite[Section~4]{Gabay83}
 and also 
 \cite[Section~10.6.4]{CP11})
applied to solve \cref{eq:prob}
consists in minimizing $\mathcal{L}$ over $b$
then over $a$ and then applying a proximal
minimization  step with respect to the 
Lagrange multiplier $u$.  
The method applied 
with a starting point  $(a_0,u_0)\in X\times X$,
generates three sequences
$(a_n)_{\nnn}$;
$(b_n)_{n\ge 1}$ and $(u_n)_{\nnn}$
via $(\forall \nnn)$:
\begin{subequations}
\label{eq:ADMM}
\begin{align}
b_{n+1}\coloneqq&(L^*L+\pt g)^{-1}(L^* a_n-L^* u_n),
\label{eq:ADMM:a}\\
a_{n+1}\coloneqq &\prox_f(Lb_{n+1}+u_n),
\label{eq:ADMM:b}\\
u_{n+1}\coloneqq & u_n+Lb_{n+1}-a_{n+1},
\label{eq:ADMM:c}
\end{align}
\end{subequations}
where $\prox_f\colon X\to X\colon x\mapsto \argmin_{y\in X} 
\left(f(y)
+\tfrac{1}{2}\norm{x-y}^2\right)$.

\begin{fact}[\thmtit{convergence of ADMM}]{\rm(see \cite[Theorem~4.1]{Gabay83}.)}
Let $(a_0,u_0)\in X \times X$,
 and let $(a_n)_{\nnn}$,
$(b_n)_{n\ge 1}$ and $(u_n)_{\nnn}$
be defined as in 
\cref{eq:ADMM}.
Then there exists $\overline{a}\in X$ 
such that 
 $a_n\weakly\overline{a}\in \argmin(f \circ L+g)$.
\end{fact}

\subsection*{The Douglas--Rachford method}
Suppose that $Y=X$ and that $L=\Id$.
In this case Problem \cref{eq:prob} becomes

 \begin{equation}
 \label{eq:prob:noL}
 \minimize_{x\in X} f(x)+g(x).
 \end{equation}

The Douglas--Rachford (DR) method, 
introduced in \cite{L-M79},
 applied to the ordered pair $(f,g)$ 
with a starting point $x_0\in X$
to solve
\cref{eq:prob:noL} generates two sequences 
$(x_n)_{\nnn}$ and $(y_n)_{\nnn}$
via:
\begin{subequations}
\label{eq:DR:all}
\begin{align}
y_{n}\coloneqq &\prox_f x_n,
\label{eq:DR:b}
\\
x_{n+1}\coloneqq &\TDR x_n,
\label{eq:DR:a}
\end{align}
\end{subequations}
where 
\begin{equation}
\label{def:TDR}
\TDR\coloneqq \TDR(f,g)
=\tfrac{1}{2}(\Id+R_gR_f)
=\Id-\prox_f+\prox_g(2\prox_f-\Id),
\end{equation}
and where 
$R_f\coloneqq 2\prox_f-\Id$.

Let $T\colon X\to X$.
Recall that the set of 
fixed points of $T$, denoted by
$\fix T$, is defined as $\fix T\coloneqq\menge{x\in X}{x=Tx}$.
\begin{fact}[\thmtit{convergence of Douglas--Rachford method}]{\rm 
(see, e.g., \cite[Theorem~1]{L-M79}
or \cite[Corollary~28.3]{BC2017}.)}
\label{fact:DRA:conv}
 Let $x_0\in X$
 and let $(x_n)_\nnn$
  and $(y_n)_\nnn$
   be defined as in \cref{eq:DR:all}.
Then there exists 
 $\overline{x}\in \fix \TDR$ 
such that $  x_n\weakly\overline{x}$
 and $y_n\weakly\prox_f\overline{x}\in \argmin(f+g)$.
\end{fact}
\subsection*{The Peaceman--Rachford method}
Let $h\colon X\to \left]-\infty,+\infty\right]$ 
be proper and let
$\phi\colon\RR_+\to \left[0,+\infty\right]$
be an increasing function that
 vanishes only at $0$. We say that $h$
 is \emph{uniformly convex} (with modulus of convexity $\phi$) if 
 $(\forall (x,y)\in \dom f\times \dom f)$ $(\forall \alpha \in \left]0,1\right[)$
 we have 
 $f(\alpha x+(1-\alpha) y)+\alpha(1-\alpha)\phi(\norm{x-y})
 \le \alpha f(x)+(1-\alpha) f(y)$.

When $g$ is uniformly convex,
the Peaceman--Rachford (PR) method, introduced in \cite{L-M79},   
can be used to solve 
\cref{eq:prob:noL}.
In this case, given
$x_0 \in X$,  PR method generates the sequences 
$(x_n)_{\nnn}$ and $(y_n)_{\nnn}$
via:
\begin{subequations}
\label{eq:PR:all}
\begin{align}
y_{n}\coloneqq &\prox_f x_n,
\\
x_{n+1}\coloneqq& \TPR x_n,
\end{align}
\end{subequations}
where 
\begin{equation}
\label{eq:def:TPR}
\TPR=\TPR(f,g)=R_gR_f=(2\prox_g-\Id)(2\prox_f-\Id).
\end{equation}

\begin{fact}[\thmtit{convergence of Peaceman--Rachford method}]{\rm 
(see, e.g., \cite[Proposition~1]{L-M79}
or \cite[Proposition~28.8]{BC2017}.)}
\label{fact:PRA:conv}
Suppose  that $g$ is uniformly 
convex.
Let $\overline{y}$ be the unique minimizer of
$f+g$, let $x_0\in X$
 and let $(x_n)_\nnn$
  and $(y_n)_\nnn$
   be defined as in \cref{eq:PR:all}.
 Then 
 $(\exists  \overline{x}\in \fix \TPR)$
 such that 
 $x_n\weakly \overline{x}$
 and $y_n\to\prox_f\overline{x}=\overline{y}$.
\end{fact}

In the sequel we use the notation
\begin{equation}
g^{\veet}\colon X\to\left]-\infty,+\infty\right]\colon x\mapsto g(-x) .
\end{equation}
Recall that the Fenchel--Rockafellar dual of \cref{eq:prob}
is
\begin{equation}
\label{eq:dual}
\minimize_{x\in X} f^*(x)+g^*(-L^*x).
\end{equation}
\begin{remark}\
\begin{enumerate}
\item
One can readily verify that 
$\pt g^\veet =(-\Id)\circ\pt g \circ(-\Id)$.
Therefore, in view of
\cite[Theorem~A]{Rock1970} and
 \cite[Lemma~3.5~on~page~125~and~Lemma~3.6~on~page~133]{EckThesis}
 (see also \cite[Corollaries~4.2~and~4.3]{JAT2012})
we have\footnote{It is straightforward to
verify that 
$g^{\veet*}= (g^*)^\veet$ (see, e.g., \cite[Proposition~13.23(v)]{BC2017}).}
\begin{equation}
\label{eq:TDR:sd}
\TDR(f, g)=\TDR(f^*,g^{*\veet}), 
\end{equation}
and 
\begin{equation}
\label{eq:TPR:sd}
\TPR(f, g)=\TPR(f^*,g^{*\veet}).
\end{equation}
\item
When $(L,Y)=(\Id, X)$,
inclusion
\cref{eq:inc:inc}
reduces to: 
Find $y\in X$ such that
$0\in \pt f  (y)+\pt g(y)$
 and the \emph{dual} inclusion (corresponding to 
the
 Fenchel--Rockafellar dual \cref{eq:dual}) is:
Find $y\in X$ such that
$0\in \pt f^*  (y)-\pt g^*(-y)=(\pt f)^{-1}y-(\pt g)^{-1}(-y)$,
which in this case coincide with the Attouch--Thera dual  
of \cref{eq:inc:inc} (see \cite{AT}). 
\end{enumerate}
\end{remark}


One can use DR method to solve 
\cref{eq:dual} where $(f,g)$ in \cref{fact:DRA:conv}
is replaced by $(f^*,g^*\circ(-L^*))$.
Recalling \cref{eq:TDR:sd} 
 we learn that $\TDR=\TDR{(f^*,g^*\circ (-L^*))}
=\TDR{(f^{**},(g^*\circ (-L^*))^{\veet *})}=\TDR{(f,(g^*\circ L^*){^*})}$,
 where the last identity follows from
  \cite[Proposition~13.44]{BC2017}

In view of \cref{def:TDR} 
 and \cref{eq:TDR:sd}
we have
\begin{equation}
\label{def:TDR:dual}
\TDR
=\TDR{(f^*,g^*\circ (-L^*))}
=\TDR(f,(g^*\circ L^*){^*}) 
=\Id-\prox_f+\prox_{(g^*\circ L^*)^*}(2\prox_f-\Id).
\end{equation}
Similarly, under additional assumptions 
(see \cref{fact:PRA:conv}), 
one can use PR method to solve 
\cref{eq:dual} where $(f,g)$ in \cref{fact:PRA:conv}
is replaced by $(f^*,g^*\circ(-L^*))$.
In this case \cref{eq:def:TPR} and 
\cref{eq:TPR:sd} imply that
\begin{equation}
\label{def:TPR:dual}
\TPR
=\TPR{(f^*,g^*\circ (-L^*))}
=\TPR(f,(g^*\circ L^*){^*}) 
=(2\prox_{(g^*\circ L^*)^*}-\Id)(2\prox_f-\Id).
\end{equation}

%

For completeness,
we
provide a concrete proof
of the formula for $\prox_{(g^*\circ L^*)^*}$
in \cref{app:A}  (see \cref{prop:JC_form} below).
We point out that the 
 formula for $\prox_{(g^*\circ L^*)^*}$
in a more general setting
is given in \cite[Proposition~4.1]{Gabay83} 
(see also \cite[Section~10.6.4]{CP11}).

%
%


\section{ADMM and Douglas--Rachford method}
\label{sec:4}
In this section we discuss the 
equivalence of ADMM and DR method.
This equivalence was first introduced 
by Gabay in \cite[Section~5.1]{Gabay83}
(see also \cite[Remark~3.14]{BK12}). 
Let $(x_0,a_0,u_0)\in X^3$.
Throughout the rest of this section, we assume that
\begin{empheq}[box=\mybluebox]{equation}
\label{eq:DR:thisec}
(x_{n+1},y_n)_\nnn=(\TDR x_n, \prox_f x_n),
\end{empheq}
where
\begin{equation}
 \label{eq:TDR:here}
\TDR\coloneqq{\TDR}_{(f,(g^*\circ L^*){^*})}
=\Id-\prox_f+L(L^*L+\pt g)^{-1}L^*(2\prox_f-\Id).
\end{equation}
Note that  the second identity in \cref{eq:TDR:here} 
follows from \cref{def:TDR:dual}
 and \cref{prop:JC_form}\ref{prop:JC_form:iv}.
We also assume that
\begin{empheq}[box=\mybluebox]{equation*}
\text{
  $(a_n,u_n,b_{n+1})_\nnn$
 is defined as in \cref{eq:ADMM}.}
\end{empheq}
The following lemma will be used later to clarify 
the equivalence of
DR and ADMM.

\begin{lem}
\label{lem:1:step:DR:ADMM}
Let $(b_{-},a_{-},u_{-})\in Y\times X\times X$
 and set 
   \begin{subequations}
   \label{eq:1stepADMM}
 \begin{align}
 (b,a,u)&\coloneqq ((L^*L+\pt g)^{-1}(L^* a_{-}-L^* u_{-}),
\prox_f(Lb+u_{-}),
  u_{-}+Lb-a),
      \label{eq:1stepADMM:a}
  \\
   (b_+,a_+,u_+)&\coloneqq ((L^*L+\pt g)^{-1}(L^* a-L^* u),
\prox_f(Lb_++u),
  u+Lb_+-a_+).
      \label{eq:1stepADMM:b}
  \end{align}
  \end{subequations}
  Then 
  \begin{subequations}
\begin{align}
  \TDR(Lb+u_{-})&=Lb_++u,
  \label{eq:1stepDR:a}
  \\
\prox_f\TDR(Lb+u_{-})&=a_+.
  \label{eq:1stepDR:b}
\end{align}
\end{subequations}

\end{lem}
\begin{proof}
Indeed, it follows from \cref{eq:TDR:here},
\cref{eq:1stepADMM},
\cref{eq:1stepDR:a}
 and \cref{eq:1stepDR:b} that
\begin{subequations}
\begin{align}
\TDR(Lb+u_{-})&=(Lb+u_{-})-\prox_f(Lb+u_{-})\nonumber\\
&\quad+L(L^*L+\pt g)^{-1}{L^*}(2\prox_f(Lb+u_{-})-(Lb+u_{-}))\\
&=(Lb+u_{-})-a+L(L^*L+\pt g)^{-1}{L^*}(2a-(Lb+u_{-}))\\
&=(Lb+u_{-})-a+L(L^*L+\pt g)^{-1}{L^*}(a-(Lb+u_{-}-a))\\
&=u+L(L^*L+\pt g)^{-1}({L^*}a-{L^*}u)\\
&=Lb_++u,
\end{align}
\end{subequations}
which proves \cref{eq:1stepDR:a}.
Now \cref{eq:1stepDR:b} 
follows from combining \cref{eq:1stepDR:a}
 and \cref{eq:1stepADMM:b}.
\end{proof}

We now prove the main result in this section by induction.
\begin{thm}
\label{Prop:ADMM:DR}
The following hold:
\begin{enumerate}
\item 
\label{i:ADMM:to:DR}
{\bf (DR as ADMM iteration)}
Using DR method with
a starting point $x_0\in X$ to solve \cref{eq:dual} is equivalent to  
using ADMM with a
starting point $(a_0,u_0)\coloneqq 
(\prox_f x_0,x_0-\prox_f x_0)$ to solve 
\cref{eq:prob}, in the sense that
$(x_n)_{n\ge1}=(Lb_n+u_{n-1})_{n\ge1}$
 and
$(y_n)_{\nnn}=(a_n)_{\nnn}$.
\item
\label{prop:DR:to:ADMM}
{\bf (ADMM as DR
 iteration)}
Using ADMM with a
starting point $(a_0,u_0)\in X\times X$ 
to solve \cref{eq:prob} is equivalent to  
using DR method with a
starting point $x_0=Lb_1+u_0$ to solve 
\cref{eq:dual}, in the sense that
$(x_n)_{\nnn}=(Lb_{n+1}+u_{n})_{\nnn}$
 and
$(y_n)_{\nnn}=(a_{n+1})_{\nnn}$.
\end{enumerate}
\end{thm}
\begin{proof}
\cref{i:ADMM:to:DR}:
Note that \cref{eq:DR:thisec}
implies that $y_0=a_0$.
Now, when $n=1$, using \cref{eq:DR:thisec}
we have 
$x_1=Tx_0=x_0-\prox_f x_0+L(L^*L+\pt g)^{-1}{L^*}(2\prox_fx_0-x_0)
=x_0-a_0+L(L^*L+\pt g)^{-1}{L^*}(2a_0-x_0)
=(x_0-a_0)+L(L^*L+\pt g)^{-1}{L^*}(a_0-(x_0-a_0))
=u_0+L(L^*L+\pt g)^{-1}{L^*}(a_0-u_0)
=u_0+Lb_1$.
Consequently, by \cref{eq:DR:thisec} we get
$y_1=\prox_f Tx_0=\prox_f x_1=\prox_f(u_0+Lb_1)=a_1$,
which verifies the base case.  
Now suppose for some $n\ge 1$
we have 
$x_n=Lb_n+u_{n-1}$
 and 
 $y_n=a_n$
  and use 
\cref{lem:1:step:DR:ADMM}
with $(b_-,a_-,u_{-})$
replaced by
$
(b_{n-1},a_{n-1},u_{n-1})$.

\cref{prop:DR:to:ADMM}:
At $n=0$, $x_0=Lb_1+u_0=Lb_{0+1}+u_0$,
 and therefore \cref{eq:DR:b} implies that
$y_0=\prox_fx_0=\prox_f(Lb_1+u_0)=a_1$
by \cref{eq:ADMM:b}.
Now suppose that for some $n\ge 0$ we have 
$
x_n=Lb_{n+1}+u_{n}
$
and
$
y_n=a_{n+1}.
$
The conclusion follows by
applying \cref{lem:1:step:DR:ADMM}
with $(b_-,a_-,u_{-})$
replaced by
$
(b_{n},a_{n},u_{n})$.
\end{proof}

\section{ADMM and Peaceman--Rachford method}
\label{sec:5}
We now turn to the equivalence of
ADMM with intermediate update of multiplier and 
PR method.
This equivalence was introduced in \cite[Section~5.2]{Gabay83}.
Given $(a_0,u_0)\in X\times X$, the ADMM 
with an \emph{intermediate} update of multiplier
applied to solve \cref{eq:prob}
generates four sequences $(a_n)_{\nnn}$,
$(u_n)_{\nnn}$,
$(b_n)_{n\ge 1}$ and  $(w_n)_{n\ge 1}$
via $(\forall \nnn)$:
\begin{subequations}
\label{eq:ADMM:all}
\begin{align}
b_{n+1}\coloneqq&(L^*L+\pt g)^{-1}(L^* a_n-L^* u_n),\label{eq:ADMM:PR:i}\\
w_{n+1}\coloneqq& u_n+Lb_{n+1}-a_{n},\label{eq:ADMM:PR:ii}\\
a_{n+1}\coloneqq&\prox_f(Lb_{n+1}+w_{n+1}),\label{eq:ADMM:PR:iii}\\
u_{n+1}\coloneqq& w_{n+1}+Lb_{n+1}-a_{n+1}.
\label{eq:ADMM:PR:iv}
\end{align}
\end{subequations}

\begin{fact}[\thmtit{convergence of ADMM with intermediate 
update of multipliers}]{\rm(see \cite[Theorem~5.3]{Gabay83}.)}
Suppose that $g$ is uniformly convex.
Let $(a_0,u_0)\in X \times X$,
 and let 
$(b_n)_{n\ge 1}$,
$(w_n)_{n\ge 1}$,
$(a_n)_{\nnn}$
and $(u_n)_{\nnn}$
be defined as in 
\cref{eq:ADMM:all}.
Then there exists $\overline{a}\in X$ 
such that 
 $a_n\to\overline{a}\in \argmin(f \circ L+g)$.
\end{fact}

In this section
we work under the additional assumption that
\begin{empheq}[box=\mybluebox]{equation*}
\text{$g$ 
is uniformly convex}.
\end{empheq}

Let $(x_0,a_0,u_0)\in X^3$.
Throughout the rest of this section we set
\begin{empheq}[box=\mybluebox]{equation}
\label{eq:it:TPR}
(x_{n+1},y_n)_\nnn=(\TPR x_n,\prox_f  x_n)_\nnn 
\end{empheq}
where
 \begin{equation}
 \label{eq:TPR:here}
\TPR\coloneqq {\TPR}_{(f,(g^*\circ L^*){^*})}
= 2L(L^*L+\pt g)^{-1}L^*(2\prox_fx_0-x_0)-2\prox_fx_0+x_0.
\end{equation}
Note that the second identity in 
\cref{eq:TPR:here} follows from \cref{def:TPR:dual}
 and \cref{prop:JC_form}\ref{prop:JC_form:iv}.
We also assume that
\begin{empheq}[box=\mybluebox]{equation*}
\text{
  $\left(a_n,u_n,b_{n+1},w_{n+1}\right)_\nnn$
 is defined as in \cref{eq:ADMM:all}.}
\end{empheq}
Before we proceed further, we prove the following 
useful lemma.
\begin{lem}
\label{lem:1:step:PR:ADMM}
Let $(b_{-},w_{-},a_{-},u_{-})\in Y\times X\times X\times X$
 and set 
 \begin{subequations}
\begin{align}
(b,w,a,u)&= ((L^*L+\pt g)^{-1}(L^* a_{-}-L^* u_{-}), u_{-}+Lb-a_{-},
\prox_f(Lb+w),
  w+Lb-a),
  \label{eq:setup:PR:a}
\\
  (b_+,w_+,a_+,u_+)
 &= ((L^*L+\pt g)^{-1}(L^* a-L^* u), u+Lb_{+}-a,
\prox_f(Lb_{+}+w_+),
  w_++Lb_{+}-a_{+}).
  \label{eq:setup:PR:b}
\end{align}
\end{subequations}
  Then 
  \begin{subequations}
\begin{align}
  \TPR(Lb+w)&=Lb_++w_+,
    \label{1step:PR:a}
    \\
\prox_f\TPR(Lb+w)&=a_+.
  \label{1step:PR:b}
\end{align}
\end{subequations}
\end{lem}
\begin{proof}
Indeed, by \cref{eq:TPR:here},
\cref{1step:PR:a}
 and \cref{1step:PR:b}
we have 
\begin{subequations}
\begin{align}
\TPR(Lb+w)
&=Lb+w-2\prox_f (Lb+w)
+2L(L^*L+\pt g)^{-1}{L^*}(2\prox_f(Lb+w)-(Lb+w))
\\
&=Lb+w-a-a
+2L(L^*L+\pt g)^{-1}{L^*}(a-(Lb+w-a))
\\
&=u-a
+2L(L^*L+\pt g)^{-1}{L^*}(a-u)=u-a
+2Lb_+
\\
&=Lb_++w_+,
\end{align}
\end{subequations}
which proves \cref{1step:PR:a}.
Now \cref{1step:PR:b} is a direct consequence of
\cref{1step:PR:a} in view of 
\cref{eq:setup:PR:b}.
\end{proof}
We are now ready for the 
main result in this section.

\begin{thm}
Suppose that $g$ is uniformly smooth.
Then the following hold:
\begin{enumerate}
\item
\label{prop:ADMM:to:PR}
{\bf (PR as ADMM iteration)}
Using PR method with a
starting point $x_0\in X$ to solve \cref{eq:dual} is equivalent to  
using ADMM with intermediate update of
multipliers with starting points $(a_0,u_0)\coloneqq 
(\prox_f x_0,x_0-\prox_f x_0)$ to solve 
\cref{eq:prob}, in the sense that
$(x_n)_{n\ge1}=(Lb_n+w_n)_{n\ge1}$
and 
$(y_n)_{\nnn}=(a_n)_{\nnn}$.
\item
\label{prop:PR:to:ADMM}
{\bf (ADMM as PR
 iteration)}
Using ADMM with
intermediate update of multiplier
with a
starting point $(a_0,u_0)\in X\times X$ to solve \cref{eq:prob} is equivalent to  
using PR method with starting point $x_0=Lb_1+w_{1}$ to solve 
\cref{eq:dual}, in the sense that
$(x_n)_{\nnn}=(Lb_{n+1}+w_{n+1})_{\nnn}$
and 
$(y_n)_{\nnn}=(a_{n+1})_{\nnn}$.

\end{enumerate}
\end{thm}
\begin{proof}
We proceed by induction.
\cref{prop:ADMM:to:PR}:
By \cref{eq:it:TPR}
and \cref{eq:ADMM:PR:i}
we have 
$x_1=\TPR x_0
=2\prox_g(2\prox_fx_0-x_0)-(2\prox_fx_0-x_0)
=x_0-2\prox_f x_0+L(L^*L+\pt g)^{-1}{L^*}(2\prox_fx_0-x_0)
=x_0-2a_0+2L(L^*L+\pt g)^{-1}{L^*}(2a_0-x_0)
=(x_0-a_0)-a_0+2L(L^*L+\pt g)^{-1}{L^*}(a_0-(x_0-a_0))
=u_0-a_0+2Lb_1
=u_0-a_0+Lb_1+Lb_1
=Lb_1+w_1$,
which verifies the base case.  
Now suppose for some $n\ge 1$
we have 
$x_n=Lb_n+w_n.$
The conclusion follows from applying \cref{lem:1:step:PR:ADMM}
with $(b_{-},w_{-},a_{-},u_{-})$ replaced 
by $(b_{n-1},w_{n-1},a_{n-1},u_{n-1})$
in view of \cref{eq:ADMM:all}.

\cref{prop:PR:to:ADMM}:
At $n=0$, the base case clearly holds.
Now suppose that for some $n\ge 0$ we have 
$x_n=Lb_{n+1}+w_{n+1}$
 and $y_n=a_{n+1}$
and use \cref{lem:1:step:PR:ADMM}
with $(b_{-},w_{-},a_{-},u_{-})$ replaced 
by $(b_{n},w_{n},a_{n},u_{n})$
in view of \cref{eq:ADMM:all}.
\end{proof}

\section{Chambolle--Pock and Douglas--Rachford methods}
\label{sec:CP}
In this section we survey the recent work by 
O'Connor and Vandenberghe \cite{OV17}
 concerning
the equivalence of Douglas--Rachford 
method and Chambolle--Pock method.
(For a detailed study of this 
correspondance in the more general framework 
of the primal-dual hybrid gradient method
and DR method with relaxation as well as connection 
to linearized ADMM we refer the reader to
\cite{OV17}.)
We work under the assumption 
 that\footnote{The assumption
that $\norm{A}\le 1$ is not restrictive.
Indeed, if $\norm{A}> 1$, one can always 
choose $\gamma\in \left]0,1/\norm{A}\right]$ 
 and work with $\gamma A$, instead of A.}
\begin{empheq}[box=\mybluebox]{equation}
\label{eq:cq:L}
\text{ $A \colon X\to Y$ is linear and that $\norm{A}\le 1$.}
\end{empheq} 

Consider the problem
 \begin{equation}
 \label{eq:prob:CP}
 \minimize_{x\in X} f(x)+g(Ax)
 \end{equation}
 and its Fenchel--Rockafellar dual
 given by
  \begin{equation}
 \label{eq:prob:CP}
 \minimize_{x\in X} f^*(-Ax)+g^*(x).
 \end{equation}
 
To proceed further, in the following we assume that
\begin{equation}
  \label{eq:cq:A:i}
\argmin(f+g\circ A) \neq \fady
\text{~and ~}
 0\in \sri(\dom g-A(\dom f)).
\end{equation}

Note that 
\cref{eq:cq:A:i}
implies that (see, e.g., \cite[Proposition~27.5(iii)(a)1]{BC2017})
 \begin{equation}
 \label{eq:connec:2:inc:CP}
\argmin(f+g\circ A)
=\zer( \pt f+\pt (g\circ A))
=\zer(\pt f+A^*\circ (\pt g) \circ A)\neq \fady.
 \end{equation}
 
In view of \cref{eq:connec:2:inc:CP},
solving \cref{eq:prob:CP} is equivalent to solving the 
inclusion:  
\begin{equation}
\label{eq:inc:inc:CP}
\text{Find $x\in X$ such that
$0\in \pt f(x)+A^*(\pt g (A x))$.}
\end{equation}

The Chambolle--Pock (CP) method
 applied 
with a staring point  $(u_0,v_0)\in X \times Y$
to solve \cref{eq:prob:CP} generates the sequences 
$({u}_n)_{\nnn}$,
and
$({v}_n)_{\nnn}$ via:
\begin{subequations}
\label{eq:CP:all}
\begin{align}
{u}_n
&=\prox_f(u_{n-1}-A^* v_{n-1}),
\label{eq:CP:a}
\\
{v}_n
&=\prox_{g^*}(v_{n-1}+A(2{u}_n-u_{n-1}))
.
\label{eq:CP:b}
\end{align}
\end{subequations}

\begin{fact}[\thmtit{convergence of Chambolle--Pock method}]{\rm 
(see \cite[Theorem~1]{CP2010}
and also
\cite[Theorem~3.1]{Cond13}.)}
\label{fact:CPA:conv}
 Let $(u_0,v_0)\in X \times Y$
 and let $(u_n)_\nnn$
  and $(v_n)_\nnn$
   be defined as in \cref{eq:CP:all}.
Then there exists 
$(\overline{u},\overline{v})\in X \times Y$
such that $  (u_n,v_n)_\nnn\weakly(\overline{u},\overline{v})$,
$\overline{u}\in \argmin(f+g\circ A)$ 
and
$\overline{v}\in \argmin(f^*\circ(-A^*)+g^*)$.
\end{fact}

It is known that the method in
\cref{eq:CP:all} reduces to DR method 
(see, e.g., \cite[Section~4.2]{CP2010}) 
when $A=\Id$. 
We state this equivalence 
in \cref{prop:DR:2:CP} below.
\begin{prop}[\thmtit{DR as a CP iteration}]
\label{prop:DR:2:CP}
Suppose that $A=\Id$.
Then, using 
DR method, defined as in \cref{eq:DR:all}, 
with a starting point $x_0\in X$ 
to solve \cref{eq:prob:CP}
is equivalent to using 
CP method with
a starting point $(u_0,v_0)\in \menge{(u,v)}{u-v=x_0}
\subseteq X\times X$ 
to solve \cref{eq:prob:CP}
in the sense that
$(x_n)_\nnn=(u_n-v_n)_\nnn$
 and 
 $(y_n)_\nnn=(u_n)_\nnn$.
\end{prop}
\begin{proof}
We use induction.
When $n=0$, the base case is
obviously true.
Now suppose that
for some $n\ge 0$
we have $x_n=u_n-v_n$
 and $y_n=u_n$.
 Then, 
in view of \cref{lem:DR:gA} below we have 
 $x_{n+1}
 =\prox_f x_n-\prox_{g^*}(2\prox_f x_n-x_n)
 =\prox_f(u_n-v_n)-\prox_{g^*}(2\prox_f (u_n-v_n)-(u_n-v_n))
 =u_{n+1}-\prox_{g^*}(v_n+2u_{n+1}-u_n)=u_{n+1}-v_{n+1}
 $. 
 The claim about $y_{n+1}$
 follows directly and the proof is complete.
\end{proof}
\subsection*{Chambolle--Pock as a DR iteration: The 
O'Connor--Vandenberghe technique}
Let $Z$ be a real Hilbert space.
In the following, we assume
 that $C\colon Z\to Y$ is linear and that
\begin{equation}
\label{eq:assmp:B}
\text{$B\colon X \times Z \to Y
\colon (x,z) \mapsto Ax+Cz$ satisfies that $BB^*=\Id$}.
\end{equation}
Note that one possible choice of $C$ is to
set $C^2\coloneqq \Id-AA^*$, where 
the existence of 
$C$ follows from, e.g., \cite[Theorem~on~page~265]{RieszNagy}.
Now consider the problem
 \begin{equation}
 \label{eq:prob:B}
 \minimize_{(x,z)\in X\times Z} 
 \widetilde{f}(x,z)+g(B(x,z)),
 \end{equation}
 where 
 \begin{equation}
 \label{eq:def:ftild}
 \widetilde{f}\colon X\times Z\to \left]-\infty,+\infty\right]
 \colon (x,z)\mapsto
f(x)+\iota_{\{0\}} (z).
 \end{equation}
The following result, proved in \cite[Section~4]{OV17} 
in the more general framework of primal-dual
hybrid gradient method,
provides an elegant way to construct
the correspondence between the DR sequence
 when applied to solve \cref{eq:prob:B}
 and the CP sequence when applied to  
solve \cref{eq:prob:CP}.
We restate the proof for the sake 
of completeness.
\begin{prop}[\thmtit{CP corresponds to a DR iteration}]
\label{prop:prop:DR:2:CP}
Using CP method with
starting point $(u_0,v_0)\in X\times Z$
 to solve \cref{eq:prob:CP}
corresponds to  
using DR with starting point $\bx_0\coloneqq 
(u_0,0)-B^*v_0\in X\times Z$ to solve 
\cref{eq:prob:B}, in the sense that
$(\bx_n)_\knn 
=((u_n,0)-B^* v_n)_{\nnn}$
 and
$(\by_n)_{\nnn}
=({u}_{n+1},0)_\nnn$.

\end{prop}

\begin{proof}
We proceed by induction.
When $n=0$, 
by assumption we have 
$\bx_0=
(u_0,0)-B^* v_0$.
It follows from \cref {prop:prod:p}\ref{prop:prod:p:i:i}\&\ref{prop:prod:p:ix}
below
that $\by_0=\prox_{\tilde{f}}\bx_0
=\prox_{\tilde{f}}((u_0,0)-(A^* v_0,C^* v_0))
=\prox_{\tilde{f}}(u_0-A^* v_0,-C^* v_0)
=(\prox_f(u_0-A^* v_0),0)$.
Now suppose that for some $n\ge 0$ we have
\begin{subequations}
\begin{align}
\bx_n&=(u_n,0)-B^* v_n,
\\
\by_n
&=({u}_{n+1},0).
\end{align}
\end{subequations}
Then 
\begin{align}
(u_{n+1},0)-B^* v_{n+1}
&=(u_{n+1},0)-B^*(\prox_{g^*}(v_{n}+A(2u_{n+1}-u_{n})))
\nonumber
\\
&=\by_n
-B^*(\prox_{g^*}(v_{n}+B(2(u_{n+1},0)-(u_{n},0))))
\nonumber
\\
&=\by_n
-B^*(\prox_{g^*}(BB^* v_{n}+B(2(u_{n+1},0)-(u_{n},0))))
\nonumber
\\
&=\by_n
-B^*\prox_{g^*}B(2(u_{n+1},0)-((u_{n},0)-B^* v_{n})))
\nonumber
\\
&=\by_n
-\prox_{(g\circ B)^*}(2\by_n-\bx_n))=\bx_{n+1},
\end{align}
where the last identity follows from 
 \cref{eq:R:DR:rho:ii} below
applied with $A$ replaced by $B$.
Now by \cref{eq:CP:a} we have 
$
(u_{n+2},0)
=(\prox_f(u_{n+1}-A^* v_{n+1}),0)
=(\prox_f(u_{n+1}-A^* v_{n+1}),0)
=\prox_{\tilde{f}} (u_{n+1}-A^* v_{n+1},-C^* v_{n+1})
=\prox_{\tilde{f}} ((u_{n+1},0)-(A^* v_{n+1},C^* v_{n+1}))
=\prox_{\tilde{f}} ((u_{n+1},0)-B^* v_{n+1})
=\prox_{\tilde{f}} \bx_{n+1}
=\by_{n+1}
$.
\end{proof}

\section{Dykstra's method and the Method of Alternating Projections}
\label{sec:6}

In this section we assume that
\begin{empheq}[box=\mybluebox]{equation*}
\label{UV:assmp}
U \text{~and~} V \text{~are nonempty closed convex subsets of~}X
\text{~such that~} U\cap V\neq \fady.
\end{empheq} 
In the sequel, we use 
$\iota_U$ to denote the \emph{indicator function} associated
with the set $U$ defined $(\forall x\in X)$ by: $\iota_U(x)=0$, 
if $x\in U$;
and $\iota_U(x)=+\infty$, otherwise.
We consider the problem 
\begin{equation}
\label{P:conv:feas}
\text{find $x\in X$ such that $x\in U\cap V$.}
\end{equation}
Note that \cref{P:conv:feas} is a special case of
\cref{eq:prob} by setting
 $(f,g,L,Y)\coloneqq (\iota_U,\iota_V,\Id,X)$.
Let $x_0\in X$ and set $p_0=q_0=0$.
Dykstra's method\footnote{We point out that
in \cite{BC08}, the authors develop a Dykstra-type method 
that extends the original method described in \cref{eq:dykstra}
to solve problems of the form \cref{eq:prob:noL}.} 
applied to solve \cref{P:conv:feas} 
 generates the sequences 
$(x_n)_{\nnn}$, $(y_n)_{\nnn}$, $(p_n)_{\nnn}$,
and $(q_n)_{\nnn}$
defined 
$(\forall \nnn)$ by
\begin{subequations}
\label{eq:dykstra}
\begin{align}
y_n&=P_V(x_n+p_n),
\label{eq:Dykstra:y}\\
p_{n+1}&=x_n+p_n-y_n,
\label{eq:Dykstra:p}\\
x_{n+1}&=P_U(y_n+q_n),
\label{eq:Dykstra:x}\\
q_{n+1}&=y_n+q_n-x_{n+1}.
\label{eq:Dykstra:q}
\end{align}
\end{subequations}
On the other hand, 
the Method of Alternating Projections (MAP)
applied to solve \cref{P:conv:feas}
generates the sequence
  $(x_n)_{\nnn}$ defined $(\forall \nnn)$ by
  \begin{equation}
 x_n=(P_UP_V)^n x_0.
  \end{equation}

 \begin{fact}[\thmtit{convergence of Dykstra's method}]{\rm (see 
 \cite[Theorem~2]{Boyle-Dykstra86}.)}
 \label{fact:Dykstra:conv}
 Let $x_0\in X$ and set $p_0=q_0=0$.
Let the sequences 
$(x_n)_{\nnn}$, $(y_n)_{\nnn}$, $(p_n)_{\nnn}$,
and $(q_n)_{\nnn}$ be defined as in \cref{eq:dykstra}.
Then 
\begin{equation}
x_n\to P_{U\cap V}x_0.
\end{equation}
 \end{fact} 
  
   \begin{fact}[\thmtit{convergence of MAP}]{\rm(see \cite{Breg65}.)}
 Let $x_0\in X$ and set $(\forall \nnn)$
$
   x_{n+1}=P_UP_Vx_n.
$
 Then 
\begin{equation}
 x_n\weakly \overline{x}\in U\cap V. 
  \end{equation}
 \end{fact}

The next result is a part of the folklore
see, e.g., \cite[comment~on~page~30]{Boyle-Dykstra86}
and also \cite[Section~9.26]{Deutsch01} 
for a general framework that involves $m$ 
closed affine subspaces, where $m\ge 2$. 
We include a simple proof in the case
of two sets for the sake of completeness.

\begin{prop}[\thmtit{Dykstra's method for two closed linear subspaces}]
\label{prop:Dykstra:MAP}
Let $x_0\in X$ and let $U$ and $V$ be closed linear subspaces of 
$X$. Set $p_0=q_0=0$.
Then the Dykstra's method generates the sequences
$(x_n)_{\nnn}$, $(y_n)_{\nnn}$, $(p_n)_{\nnn}$,
and $(q_n)_{\nnn}$
defined 
$(\forall \nnn)$ by
\begin{subequations}
\label{eq:Dykstra:subsp}
\begin{align}
y_n&=P_Vx_n,
\label{eq:Dykstra:subsp:y}\\
p_{n+1}&=P_{V^{\perp}}\sum_{k=0}^{n}x_k,
\label{eq:Dykstra:subsp:p}\\
x_{n+1}&=P_Uy_n,
\label{eq:Dykstra:subsp:x}\\
q_{n+1}&=P_{U^{\perp}}\sum_{k=0}^{n}y_k.
\label{eq:Dykstra:subsp:q}
\end{align}
\end{subequations}
Consequently, Dykstra's method in this case 
is eventually MAP in the sense that
$(\forall \nnn) $ $x_{n+1}=P_UP_Vx_n$.
\end{prop}
\begin{proof}
At $n=0$ the base case is obviously true.
Now suppose that for some $n\ge 0$ we have 
\cref{eq:Dykstra:subsp} holds.
By \cref{eq:Dykstra:y} and \cref{eq:Dykstra:subsp:p}
we have $y_{n+1}=P_V(x_{n+1}+p_{n+1})
=P_V(x_{n+1}+P_{V^{\perp}}\sum_{k=0}^{n}x_k)
=P_Vx_{n+1}$.
Moreover, \cref{eq:Dykstra:p}
and \cref{eq:Dykstra:subsp:p}
imply that $p_{n+2}=x_{n+1}+p_{n+1}-y_{n+1}=x_{n+1}+
P_{V^{\perp}}\sum_{k=0}^{n}x_k-P_Vx_{n+1}=
P_{V^{\perp}}x_{n+1}+P_{V^{\perp}}\sum_{k=0}^{n}x_k
=P_{V^{\perp}}\sum_{k=0}^{n+1}x_k$, as claimed.
The statements for $x_{n+2}$ and $q_{n+2}$ are proved similarly.
\end{proof}

\begin{remark}[\thmtit{Dykstra vs.\ ADMM, DR and CP}]\
\begin{enumerate}
\item
\label{rem:ADMM:DR:Dykstra}
In view of 
\cref{Prop:ADMM:DR}\cref{i:ADMM:to:DR}\&\cref{prop:DR:to:ADMM},
ADMM is equivalent to DR method 
whereas the latter is not equivalent to MAP (see, e.g., \cite{JAT2014}).
Therefore, in view of \cref{prop:Dykstra:MAP},
we conclude that Dykstra's method is \emph{neither equivalent} to 
ADMM \emph{nor} to DR method.
\item 
\label{rem:ADMM:Dykstra}
In the special case when $x_0\in U\cap V$, where 
$U$ and $V$ are nonempty closed convex subsets
such that $U\cap V\neq \fady$, both DR with starting point $x_0$ 
(equivalently, in view of \cref{i:ADMM:to:DR},
 ADMM with starting point\footnote{In passing, we mention that
when $f=\iota_U$ we have $\prox_f=P_U$,
see, e.g., \cite[Example~23.4]{BC2017}.} 
 $(a_0,u_0)=(x_0,x_0-P_Ux_0)$)
 and Dykstra's method with starting point $x_0$
will converge after one iteration to $x_0$.
Therefore, we conclude that
if we start at a \emph{solution}, 
then the two methods generate the same sequences 
(see \cite[Section~5.1.1]{Boyd-PCPE11}).  
\item
\label{rem:ADMM:CP:Dykstra}
Similar to the argument in  \cref{rem:ADMM:DR:Dykstra},
in view of 
\cref{prop:DR:2:CP} and \cref{prop:prop:DR:2:CP}
one can conclude that, in general,
 CP method 
 \emph{neither equivalent} to 
MAP \emph{nor} to Dykstra's method.
\end{enumerate}
\end{remark}

It is well-known that the
 equivalence of Dykstra's method 
 and MAP may fail in general if 
we remove the assumption that both $U$ and 
$V$ are closed linear subspaces,
see, e.g., \cite[Figure~30.1]{BC2017}.
We provide another example below where one set set is
a linear subspace and the other set is a half-space.

\begin{example}[\thmtit{Dykstra's method vs.\ MAP}]
\label{ex:Dykstra vs. MAP}
Suppose that $X=\RR^2$,
that 
$U=\RR\cdot(1,1)$
and that
$V=\RR\times\RR_{-}$. 
Let $(\alpha,\beta)\in \RR_{--}\times\RR_{++}$ 
such that $\beta\le-\alpha$, 
let $x_0=(\alpha,\beta)$ and set $p_0=q_0=(0,0)$.
Let $(x_n)_\nnn$ be the sequence generated by 
Dykstra method \cref{eq:dykstra}.
Then MAP produces the sequence $((P_UP_V)^n x_0)_\nnn$
and one readily verifies that
$(\forall n\ge 1)$ 
$(P_UP_V)^n x_0=(\tfrac{1}{2}\alpha,\tfrac{1}{2}\alpha)$.
Hence, in view of \cref{fact:Dykstra:conv} we have 
\begin{equation}
(P_UP_V)^n x_0\to(\tfrac{1}{2}\alpha,\tfrac{1}{2}\alpha)
\neq (\tfrac{1}{2}(\alpha+\beta),\tfrac{1}{2}(\alpha+\beta))
= P_{U\cap V} x_0\leftarrow x_n.
\end{equation}

\end{example}

\begin{figure}[h!]
\center{
\includegraphics[scale=0.16]{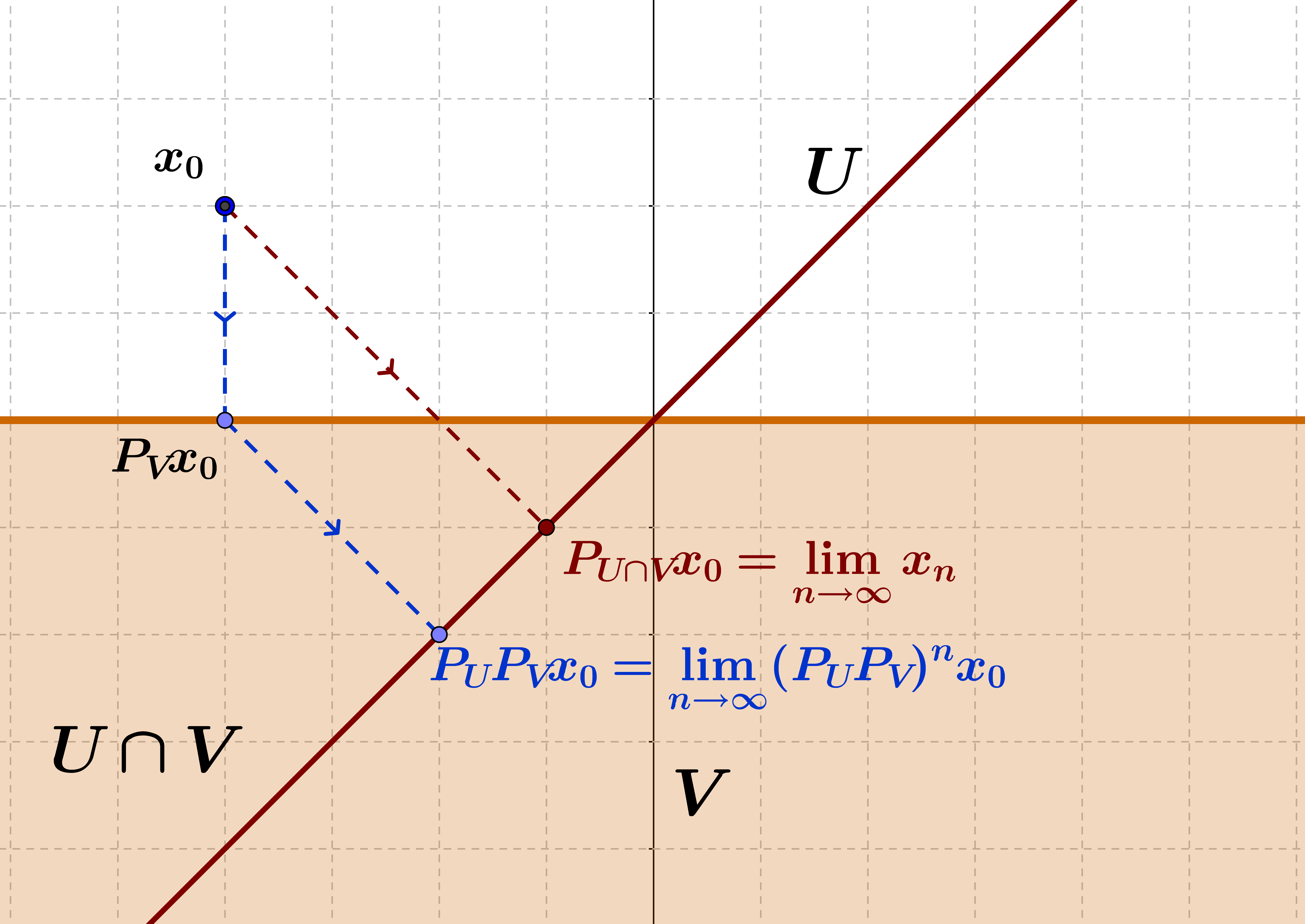}}
\caption{A \texttt{GeoGebra} \cite{geogebra} 
snapshot that illustrates \cref{ex:Dykstra vs. MAP}.
$U$ is the red line and $V$ is the shaded half-space.
Shown also are the starting point $x_0$,
the iterates of the 
MAP sequence $((P_UP_V)^n x_0)_\nnn$
(which is eventually constant
and equals to $\lim_{n\to \infty}(P_UP_V)^n x_0$) 
and the limit $P_{U\cap V}x_0$ 
of the sequence $(x_n)_\nnn$ produced by 
Dykstra's method.
}
\end{figure}
\begin{remark}[\thmtit{forward-backward vs.\ ADMM}]
One can use the forward-backward method
to solve \cref{P:conv:feas} by setting 
$(f,g)=(\tfrac{1}{2}d^2_U,\iota_V)$.
In this case the method reduces to MAP
(see, e.g., \cite[Example~10.12]{CP11}) 
which is not DR (see, e.g., \cite{JAT2014}).
Therefore, in view of 
\cref{Prop:ADMM:DR}\cref{i:ADMM:to:DR}\&\cref{prop:DR:to:ADMM} 
we conclude that 
the forward-backward method
is \emph{not} equivalent to ADMM.
\end{remark}
\section*{Acknowledgments}
WMM was supported by the Pacific Institute of Mathematics 
Postdoctoral Fellowship and 
the DIMACS/Simons 
Collaboration on Bridging
Continuous and Discrete 
Optimization through NSF grant \# CCF-1740425.

\begin{appendices}
	
\crefalias{section}{appsec}
\section{}\label{app:A}



 Let $A\colon X\to X$ be linear. Define
\begin{equation}
q_A\colon X\to \RR\colon x\mapsto\tfrac{1}{2}\innp{x, Ax}.
\end{equation}
 
  Recall that a linear operator $A\colon X\to X$ 
 is \emph{monotone} if 
 $(\forall x\in X)$ $\innp{x, Ax}\ge 0$, and 
 is \emph{strictly monotone} if 
 $(\forall x\in X\smallsetminus \{0\})$ $\innp{x, Ax}> 0$.
 Let $h\colon X\to \RR$ and let $x\in X$.
We say that $h$ is \emph{Fr\'{e}chet 
differentiable at $x$} if there exists a linear operator 
$Dh(x)\colon X\to \RR$, called the 
   \emph{Fr\'{e}chet 
derivative} of $h$  at $x$,
such that
$\lim_{0\neq \norm {y}\to 0}\frac{h(x+y)-h(x)-Dh(x)y}{\norm{y}}=0$;
and $h$ is \emph{Fr\'{e}chet 
differentiable on $X$}  if it is Fr\'{e}chet 
differentiable at every point in $X$.
Also, we say that $h$ is \emph{G\^{a}teaux 
differentiable at $x$} if there exists a linear operator 
$Dh(x)\colon X\to \RR$, called the 
   \emph{G\^{a}teaux
derivative} of $h$  at $x$,
such that $(\forall y\in X)$
$Dh(x)y=\lim_{\alpha\downarrow 0}\frac{h(x+\alpha y)-h(x)}{\alpha}$;
and $h$ is \emph{G\^{a}teaux
differentiable on $X$}  if it is G\^{a}teaux
differentiable at every point in $X$.

 The following lemma is a special case of
 \cite[Proposition~17.36]{BC2017}. 
 \begin{lemma}
 \label{lem:qA:adapt}
 Let $A\colon X\to X$ be linear, strictly monotone, self-adjoint and invertible.
 Then the following hold:
 \begin{enumerate}
 \item
  \label{lem:qA:adapt:i}
 $q_A$ and $q_{A^{-1}}$ are strictly convex, continuous,
 Fr\'{e}chet differentiable 
 and $(\grad q_A,\grad q_{A^{-1}} )=(A,A^{-1})$.
 \item
   \label{lem:qA:adapt:ii}
 $q_{A}^*=q_{A^{-1}}$.
 \end{enumerate}
 \end{lemma}
 \begin{proof}
 Note that, likewise $A$, $A^{-1}$ is linear, strictly monotone, self-adjoint (since $(A^{-1})^*
 =(A^{*})^{-1}=A^{-1}$) and invertible.
Moreover,  $\ran A=\ran A^{-1}=X$.
 \ref{lem:qA:adapt:i}:
 This follows from \cite[Example~17.11~and~Proposition~17.36(i)]{BC2017}
 applied to $A$ and 
$A^{-1}$ respectively.
 \ref{lem:qA:adapt:ii}: It follows from 
 \cite[Proposition~17.36(iii)]{BC2017},
 \cite[Theorem~4.8.5.4]{StoerBulirsch02}
and the invertibility of $A$ that
 $q_{A}^*=q_{A^{-1}}+\iota_{\ran A}
 =q_{A^{-1}}+\iota_{X}=q_{A^{-1}}$.
 \end{proof}
 
\begin{prop}
\label{prop:JC_form}
Let 
$L\colon Y \to X$ be linear.
Suppose 
that $L^*L$ is invertible. Then the following hold:
\begin{enumerate}
\item
\label{prop:JC_form:0:-1}
$\ker L=\{0\}$ .
\item
\label{prop:JC_form:0}
$L^*L$ is strictly monotone.
\item
\label{prop:JC_form:i}
$\dom(q_{L^*L}+g)^*=X$.
\item
\label{prop:JC_form:i:-1}
$\pt(q_{L^*L}+g)=\grad q_{L^*L}+\pt g=L^*L+\pt g$.
\item
\label{prop:JC_form:i:0}
$(q_{L^*L}+g^*)^*$ is  Fr\'{e}chet 
differentiable on $X$.
\item
\label{prop:JC_form:i:i}
$(L^*L+\pt g^*)^{-1}$ is single-valued and 
$\dom (L^*L+\pt g^*)^{-1}=X$.
\item
\label{prop:JC_form:ii}
$\prox_{g^*\circ L^*}=\Id-L(L^*L+\pt g)^{-1}L^*$.
\item
\label{prop:JC_form:iv}
$\prox_{(g^*\circ L^*)^{*}}=L(L^*L+\pt g)^{-1}L^*$.
\end{enumerate}

\end{prop}
\begin{proof}
\ref{prop:JC_form:0:-1}:
Using
\cite[Fact~2.25(vi)]{BC2017}
 and the assumption that 
 $L^*L$ is invertible we have 
$\ker L=\ker L^*L=\{0\}$.
\ref{prop:JC_form:0}:
Using \ref{prop:JC_form:0:-1}
we have  $(\forall x\in X\smallsetminus\{0\})$
$\innp{L^*Lx,x}=\innp{Lx,Lx}=\normsq{Lx}>0$, hence
$L^*L$ is strictly monotone.
\ref{prop:JC_form:i}:
By \ref{prop:JC_form:0} and \cref{lem:qA:adapt}\ref{lem:qA:adapt:i} 
applied with $A$ replaced by $L^*L$
we have $\dom q_{L^*L}=\dom q^*_{L^*L}=X$, 
hence 
\begin{equation}
\label{eq:cq:q*}
\dom q_{L^*L}- \dom g=X- \dom g=X. 
\end{equation}
It follows from \cref{eq:cq:q*},
\cite[Corollary~2.1]{AB86}
 and \cref{lem:qA:adapt}\ref{lem:qA:adapt:ii}\&\ref{lem:qA:adapt:i}
  that 
 $\dom (q_{L^*L}+g)^*=\dom q_{L^*L}^*+\dom g^{*}
=\dom q_{(L^*L)^{-1}}+\dom g^*=X+\dom g^*=X$.
\ref{prop:JC_form:i:-1}:
Combine \cref{eq:cq:q*},
\cite[Corollary~2.1]{AB86}
and \cref{lem:qA:adapt}\ref{lem:qA:adapt:i}.
\ref{prop:JC_form:i:0}:
Since $q_{L^*L}$ is strictly convex,
so is $q_{L^*L}+g$, which in view of
\cite[Proposition~18.9]{BC2017}
 and \ref{prop:JC_form:i}
 implies that
$(q_{L^*L}+g)^*$ is G\^{a}teaux 
differentiable
on $X=\intr \dom (q_{L^*L}+g)^*$. 
\ref{prop:JC_form:i:i}:
Using \ref{prop:JC_form:i:-1}, \cref{fact:func:sd}\ref{eq:sb:inv:conj} 
applied with $f$ replaced by $q_{L^*L}+g$,
\ref{prop:JC_form:i:0}
 and \cite[Proposition~17.31(i)]{BC2017}
we have $(L^*L+\pt g)^{-1}=(\pt(q_{L^*L}+g))^{-1}
=\pt(q_{L^*L}+g)^*
=\{\grad(q_{L^*L}+g)^*\}$ is single-valued with 
$\dom (L^*L+\pt g)^{-1}=X$.
\ref{prop:JC_form:ii}:
Let $x\in X=\dom (L^*L+\pt g)^{-1}$ 
and let $y\in X$ such that $y=x-L(L^*L+\pt g)^{-1}L^*x$.
Then using \ref{prop:JC_form:i:i} we have 
\begin{equation}
\label{eq:awl}
x=y+Lu\quad\text{where}\quad u=(L^*L+\pt g)^{-1}L^*x.
\end{equation}
Consequently, 
 $L^*y+L^*Lu=L^*x \in L^*Lu+\pt g(u)$,
hence
 $L^*y\in \pt g (u)$,
 equivalently, in view of 
 \cref{fact:func:sd}\ref{eq:sb:inv:conj} 
 applied with $f$ replaced by $g$,
 $u\in (\pt g)^{-1}(L^*y)=\pt g^{*} (L^*y)$.
Combining with \cref{eq:awl}
 we learn that 
 \begin{equation}
 \label{eq:interm}
x\in y+L\circ(\pt  g^{*})\circ L^* (y).
\end{equation}
Note that \cite[Fact~2.25(vi)~and~Fact~2.26]{BC2017}
implies that $\ran L^*=\ran L^*L=X$, hence
$0\in \sri(\dom g^*-\ran L^*)$.
Therefore one can apply
\cite[Corollary~16.53(i)]{BC2017}
to re-write \cref{eq:interm} as
$x\in(\Id+\pt (g^*\circ L^*))y$.
Therefore,
 $y=\prox_{g^*\circ L^*}x$
 by \cite[Proposition~16.44]{BC2017}.
\ref{prop:JC_form:iv}:
Apply \cref{fact:func:sd}\ref{fact:func:JA:prox}
with $f$ replaced by $g^*\circ L^*$.
\end{proof}

\crefalias{section}{appsec}
\section{}\label{app:B}

In the following, we make use of the 
useful fact (see \cite[Theorem~3.5.5]{Zali06})
\begin{equation}
\label{eq:eqv:us:uc}
\text{$f$ is uniformly convex $\siff$ $f^*$ is uniformly smooth.}
\end{equation}
Let 
$h\colon X\to \RR$. 
We say that $h$
 is \emph{uniformly smooth} if there exists
 a function $\phi\colon \RR_+\to\left[0,+\infty\right]$
 (called the modulus of uniform smoothness)
 that
 vanishes 
 at $0$ such that
 $(\forall (x,y)\in \dom f\times \dom f)$ 
 $(\forall \alpha \in \left]0,1\right[)$
 we have 
 $f(\alpha x+(1-\alpha) y)+\alpha(1-\alpha)\phi(\norm{x-y})
 \ge \alpha f(x)+(1-\alpha) f(y)$.

\begin{lem}
\label{lem:con:smooth}
The following hold:
\begin{enumerate}
\item
\label{lem:con:smooth:i}
$g^*$ is uniformly smooth. 
\item
\label{lem:con:smooth:ii}
$g^*\circ L^*$ is uniformly smooth.
\item
\label{lem:con:smooth:iii}
$(g^*\circ L^*)^*$ is uniformly convex.
\end{enumerate}
\end{lem}
\begin{proof}
\ref{lem:con:smooth:i}:
Apply \cref{eq:eqv:us:uc} with $f$
replaced by $g$.
\ref{lem:con:smooth:ii}:
Suppose that $\phi$ is the modulus of
 uniform convexity of $g$.
Using \cite[Proposition~2.6(ii)]{Aze-P95} 
we learn that
$g^*$ is uniformly smooth with modulus 
 $\phi^\#\colon \left[0,+\infty\right[\to\left[ 0,+\infty\right]
 \colon 
 s\mapsto \sup\menge{rs-\phi (r)}{r\ge 0}$.
 Moreover  \cite[Comment~on~page~708]{Aze-P95}  
 implies that  $\phi^\#$
 is a nondecreasing function
that vanishes only at $0$.
Set $\psi\coloneqq \phi^\#(\norm{L^*}\cdot)$
and note that \cite[Fact~2.25(ii)]{BC2017}
 and 
\cref{prop:JC_form}\ref{prop:JC_form:0:-1}
imply that $\norm{L^*}=\norm{L}\neq 0$.
Consequently, likewise $\phi^\#$,
 $\psi$  is a nondecreasing function
that vanishes only at $0$.
Let $(x,y)\in \dom f\times \dom f$ 
 and let $ \alpha \in \left]0,1\right[$.
 We claim that 
\begin{equation}
(g^*\circ L^*)(\alpha x+(1-\alpha) y)+\alpha(1-\alpha)\psi(\norm{x-y})
 \ge \alpha (g^*\circ L^*)(x)+(1-\alpha) (g^*\circ L^*)(y).
 \end{equation}
 Indeed, we have
$(g^*\circ L^*)(\alpha x+(1-\alpha) y)+\alpha(1-\alpha)\psi(\norm{x-y})
=g^*(\alpha L^* x+(1-\alpha) L^*y)+\alpha(1-\alpha)\phi(\norm{L^*}\norm{x-y})
\ge g^*(\alpha L^* x+(1-\alpha) L^*y)+\alpha(1-\alpha)\phi(\norm{L^*x-L^*y})
\ge  \alpha g^*(L^* x)+(1-\alpha)g^*( L^*y)
=\alpha (g^*\circ L^*)(x)+(1-\alpha) (g^*\circ L^*)(y)$;
equivalently $g^*\circ L^*$ is uniformly smooth
with modulus $\psi$.
\ref{lem:con:smooth:iii}:
Combine \ref{lem:con:smooth:ii}
and \cref{eq:eqv:us:uc} applied with 
$f$ replaced by $g^*\circ L^*$.
\end{proof}
	
	\crefalias{section}{appsec}
	\section{}\label{app:C}
	We start by recalling the following well-known fact.
\begin{fact}
\label{fact:func:sd}
Let $f\colon X\to \left]-\infty,+\infty\right]$ 
be convex, lower semicontinuous and proper.
Then the following hold:
\begin{enumerate}
\item
\label{eq:sb:inv:conj} 
$(\pt f)^{-1}=\pt f^*.$
\item
\label{fact:func:JA:prox}
$\prox_f+\prox_{f^*}=\Id$.
\end{enumerate}
\end{fact}
\begin{proof}
\ref{eq:sb:inv:conj}: See, e.g., \cite[Remark~on~page~216]{Rock1970}
or \cite[Th\'{e}or\`{e}me~3.1]{Gossez}. 

\ref{fact:func:JA:prox}: See, e.g., \cite[Theorem~14.3(iii)]{BC2017}.
\end{proof}

	\begin{lem}
\label{lem:DR:gA}
The Douglas--Rachford method 
given in \cref{eq:DR:all}
applied to
the ordered pair $(f,g)$ 
with a starting point $x_0\in X$
to solve
\cref{eq:prob:noL}
can be rewritten as:			
\begin{subequations}
\label{eq:R:DR:rho}
\begin{align}
y_n&= \prox_{f} x_{n}
\label{eq:R:DR:rho:i}
\\
x_{n+1}&=y_{n}-\prox_{g^*}(2y_n-x_{n}).
\label{eq:R:DR:rho:ii}
\end{align}
\end{subequations}
\end{lem}
\begin{proof}
Using \cref{def:TDR}, 
and 
\cref{fact:func:sd}\ref{fact:func:JA:prox}
applied with $f$
replaced by $g$
we have
\begin{align}
x_{n+1}&=
x_n-\prox_f x_n+\prox_g(2\prox_f x_n-x_n)
=x_n-y_n+\prox_g(2y_n-x_n)
\nonumber
\\
&=x_n-y_n+2y_n-x_n-\prox_{g^*}(2y_n-x_n)=y_{n}-\prox_{g^*}(2y_n-x_{n}),
\end{align}
and the conclusion follows.
\end{proof}
	
\crefalias{section}{appsec}
\section{}\label{app:D}	
\begin{prop}
\label{prop:prod:p}
Let $(x,y,z)\in X\times Y\times Z$.
Then the following hold:
\begin{enumerate}
\item
\label{prop:prod:p:i:i}
$B^* y=(A^* y,C^* y)$.
\item
\label{prop:prod:p:ii}
$\dom  \widetilde{f}=\dom f\times \{0\}$.

\item
\label{prop:prod:p:iii}
$(\forall (x,z)\in\dom  \widetilde{f})$ we have 
$z=0$ and 
$B(x,z)=Ax$.
\item
\label{prop:prod:p:iv}
$B(\dom  \widetilde{f})=A(\dom f)$.

\item
\label{prop:prod:p:vii}
$0\in \sri (\dom g-B(\dom  \widetilde{f}))$.


\item
\label{prop:prod:p:v}
$\argmin(\widetilde{f}+g\circ B)=\argmin (f+g\circ A)\times \{0\} \neq \fady$.

\item
\label{prop:prod:p:viii}
$\prox_{\widetilde{f}}(x,z)=(\prox_f x,0)$.
\item
\label{prop:prod:p:ix}
$\prox_{(g\circ B)^*}=B^*\prox_{g^*}B$.
\end{enumerate}
\end{prop}
\begin{proof}
\ref{prop:prod:p:i:i}:
Clear.
 \ref{prop:prod:p:ii}:
It follows from \cref{eq:def:ftild} that
$\dom\widetilde{f}=\dom f\times \dom\iota_{\{0\}}=\dom f\times\{0\}$.
\ref{prop:prod:p:iii}:
The claim that $z=0$ follows from \ref{prop:prod:p:ii}.
Now combine with \cref{eq:assmp:B}.
\ref{prop:prod:p:iv}:
Combine \ref{prop:prod:p:ii} and \ref{prop:prod:p:iii}.
\ref{prop:prod:p:vii}:
Combine
\ref{prop:prod:p:iv}
and 
\cref{eq:cq:A:i}.
\ref{prop:prod:p:v}:
It follows from \ref{prop:prod:p:iii},
\ref{prop:prod:p:vii}
 and \cref{eq:connec:2:inc}
 applied with $A$ replaced by $B$
 that 
$\argmin(\widetilde{f}+g\circ B)
=\zer (\pt \widetilde{f}+B^* \circ \pt g\circ B)
=\zer (\pt f \times N_{\{0\}}+((A^* \circ \pt g\circ A)\times(C^* \circ \pt g\circ A)) )
=(\zer(\pt f+A^* \circ \pt g\circ A)\times(\zer(N_{\{0\}}+C^* \circ \pt g\circ A)))
$.
Therefore, $(x,z)\in \argmin(\widetilde{f}+g\circ B)$
$\siff$ [$ z=0$ and $x\in \zer(\pt f+A^* \circ \pt g\circ A)$]
$\siff $ $(x,z)\in\argmin(f+g\circ A) \times \{0\}$.
No combine with \cref{eq:connec:2:inc}.
\ref{prop:prod:p:vii}:
Combine \cref{eq:def:ftild}
 and \cite[Proposition~23.18]{BC2017}.
\ref{prop:prod:p:ix}:
Apply
\cref{prop:JC_form}\ref{prop:JC_form:iv}
with $(g,L)$ replaced by $(g^*, B^*)$
 and use \cref{eq:assmp:B}.
\end{proof}

\end{appendices}	
%
%

%
%
\end{document}